\documentclass[11pt]{amsart}

\usepackage{amssymb} \usepackage{amscd} \usepackage{amsmath}

\usepackage{color}

\begin{document}

\newtheorem{theorem}{Theorem}[section]
\newtheorem{conjecture}[theorem]{Conjecture}
\newtheorem{corollary}[theorem]{Corollary}
\newtheorem{definition}[theorem]{Definition}
\newtheorem{lemma}[theorem]{Lemma}
\newtheorem{proposition}[theorem]{Proposition}
\newtheorem{remark}[theorem]{Remark}

   \newtheorem{thm}{Theorem}[section]
   \newtheorem{cor}[thm]{Corollary}
   \newtheorem{lem}[thm]{Lemma}
   \newtheorem{prop}[thm]{Proposition}
   \newtheorem{defi}[thm]{Definition}
   \newtheorem{rem}[thm]{Remark}
   \newtheorem{example}[thm]{Example}

\newcommand{\be}{\begin{equation}} \newcommand{\ee}{\end{equation}}
\newcommand{\beq}{\begin{eqnarray}}
\newcommand{\eeq}{\end{eqnarray}}
\newcommand{\beqy}{\begin{eqnarray*}}
\newcommand{\eeqy}{\end{eqnarray*}}
\newcommand{\parag}{\medskip\noindent} \newcommand{\ov}{\overline}
\newcommand{\bl}{\begin{lem}} \newcommand{\el}{\end{lem}}

\newcommand{\bt}{\begin{theorem}} \newcommand{\et}{\end{theorem}}
\newcommand{\bc}{\begin{cor}} \newcommand{\ec}{\end{cor}}
\newcommand{\bit}{\begin{itemize}} \newcommand{\eit}{\end{itemize}}
\newcommand{\bin}{\begin{enumerate}}
\newcommand{\ein}{\end{enumerate}}

\newcommand{\capa}{{\rm cap}}

\newcommand{\Lip}{{\rm Lip}} \newcommand{\supp}{{\rm supp \;}}
\newcommand{\sign}{{\rm \, sign \;}} \newcommand{\R}{\mathbb{R}}
\newcommand{\N}{\mathbb{N}} \newcommand{\Z}{\mathbb{Z}}
\newcommand{\T}{\mathbb{T}} \newcommand{\Q}{\mathbb{Q}}
\newcommand{\C}{\mathbb{C}} \newcommand{\loc}{\text{loc}}
\newcommand{\weak}{{(\rm w)}}
\newcommand{\norm}[1]{\left\Vert#1\right\Vert}
\newcommand{\brkt}[1]{\left(#1\right)}
\newcommand{\abs}[1]{\left|#1\right|}
\newcommand{\set}[1]{\left\{#1\right\}}
\newcommand{\esc}[1]{\langle{#1}\rangle}

\def\R{{\mathbb{R}}} \def\oo{{\infty}} \def\a{{\alpha}}
\def\b{{\beta}} \def\g{{\gamma}} \def\d{{\delta}}
\def\e{{\varepsilon}} \def\z{{\zeta}} \def\th{{\theta}}
\def\l{{\Pi}} \def\r{{\varrho}} \def\s{{\sigma}} \def\t{{\tau}}
\def\f{{\varphi}} \def\p{{\psi}} \def\w{{\omega}} \def\G{{\Gamma}}
\def\D{{\Delta}} \def\Th{{\Theta}} \def\L{{\Pi}}
\def\S{{\Sigma}} \def\F{{\Phi}} \def\O{{\Omega}} \def\P{{\Psi}}
\def\na{{\nabla}} \def\Cap{{\rm Cap}}

\def\Nat{{\mathbb{N}}}

\def\sop{{\rm\,\mbox{\rm supp}\,}} \def\sgn{{\rm\mbox{\rm sgn}}}
\def\loc{\rm loc} \def\osc{{\rm osc}} \def\dim{{\rm\, \mbox{\rm
dim}\,}} \def\Per{{\rm \mbox{\rm Per}}}



\title{Factorizations of weighted Hardy inequalities}
\author{Sorina Barza, Anca N. Marcoci, Liviu G. Marcoci}
\address{Department of Mathematics and Computer Science, Karlstad University, SE-65188 Karlstad, SWEDEN {\em E-mail: sorina.barza@kau.se}}
\address{Department of Mathematics and Computer Science, Technical University of Civil Engineering Bucharest, RO-020396 Bucharest, ROMANIA
	{\em E-mail: li\-viu.marcoci@utcb.ro}}
\address{Department of Mathematics and Computer Science, Technical University of Civil Engineering Bucharest, RO-020396 Bucharest, ROMANIA
	{\em E-mail: an\-ca.marcoci@utcb.ro}}

\keywords{factorization of function spaces, Hardy averaging integral operator, Ces\` aro
spaces, Copson spaces}

\subjclass[2010]{Primary 26D20; Secondary 46B25}

\begin{abstract}
	We present factorizations of weighted Lebesgue, Ce\-s\` aro and Copson spaces, for weights satisfying the conditions which assure the boundedness of the Hardy's integral operator between weighted Lebesgue spaces. 
	Our results enhance, among other,  the best known forms of weighted Hardy inequalities. 
\end{abstract}

\maketitle

\section{Introduction}

Let $p>1$. The classical Hardy inequalities \be
\left(\sum_{n=1}^{\infty}\left(\frac{1}{n}\sum_{k=1}^{n} |a_k|\right)^p\right)^{1/p}\le
\frac{p}{p-1} \left(\sum_{n=1}^{\infty} |a_n|^p\right)^{1/p}\ee and
\be
\left(\int_{0}^{\infty}\left(\frac{1}{x}\int_{0}^{x}|f(t)|dt\right)^pdx\right)^{1/p}\le
\frac{p}{p-1} \left(\int_{0}^{\infty}| f(t)|^p\right)^{1/p}\ee (see
e.g.\cite{KMP}) can be interpreted as inclusions between the
Lebesgue space and Ces\` aro space of sequences (respectively
functions).
The Ces\` aro space of sequences is defined to be the set of all real sequences $a=(a_n)_{n\ge 1}$ that satisfy
$$
\|a\|_{\text{ces}(p)}=\left(\sum_{n=1}^{\infty}\left(\frac{1}{n}\sum_{k=1}^{n} |a_k|\right)^p\right)^{1/p}<\infty
$$
and  the Ces\` aro space of functions  is defined to be the set of all Lebesgue measurable real functions on $[0,\infty)$ such that
$$
\|f\|_{\text{Ces}(p)}=\left(\int_{0}^{\infty}\left(\frac{1}{x}\int_{0}^{x}|f(t)|dt\right)^pdx\right)^{1/p}<\infty.
$$

The same
interpretation is valid if the Hardy operator is substituted by
its dual. In his celebrated book \cite{B}, G. Bennet
''enhanced'' the classical Hardy inequality by substituting it
with an equality, factorizing the Ces\` aro space of sequences, with the final aim to characterize its K\"{o}the dual. He proved that a sequence $x$
belongs to the Ces\` aro space of sequences $ces(p)$ if and only if
it admits a factorization $x=y\cdot z$ with $y\in l^p$ and
$z_1^{p'}+\ldots z_n^{p'}=O(n)$, where $p'=\frac{p}{p-1}$ is
the conjugate index of $p$. This factorization  gives also a better insight in the structure of Ces\` aro spaces. 
The answer to the question of findind the dual space of the Ces\` aro space was given for the first time by A. Jagers \cite{J}.  This problem was posed by the Dutch Academy of Sciences.
His conditions are nevertheless difficult to check. By means of factorizations a new isometric characterization, an alternative description of the dual Ces\` aro space of sequences was given by G. Bennett in \cite{B}.

In the case of functions, the same factorization results as well as the dual space of Ces\` aro space are only  mentioned in
\cite[Ch. 20]{B}, for the unweighted spaces. A factorization result for the unweighted Ces\` aro function spaces was proved in \cite{AM} 
where also an isomorphic description of the dual space of the Ces\` aro space of functions was given. An isometric description for the general weighted case,
in the spirit of A. Jagers was given in \cite{K}. Although the characterization is given for general weights, the condition is difficult to checked. 
The results of \cite{B} have a big
impact in many parts of analysis but it seems that the
corresponding results for weighted spaces are less studied. However, in \cite{CH} can be found the following factorization result which is a weighted integral analogue of a result obtained by G. Bennett in \cite{B} for the discrete Hardy operator in the unweighted case.
\bt[\cite{CH}]
Let $1<p<\infty$ and $w$, $v$ two weights such that $w>0$, $v>0$ a.e. and assume that $h$ is a non-negative function on $[0,\infty)$.
Then the function $h$ belongs to $Ces_p(w)$ if and only if it admits a
factorization $h=f \cdot g$,  $f\ge 0$, $g>0$ on $[0,\infty),$ with $f\in L^p(v)$ and $g$ such that 
$$
\|g\|_{w,v}=\sup_{t>0}\left(\int_t^{\infty}\frac{w(s)}{s^p}ds\right)^{1/p}\left(\int_0^{t}v^{1-p'}(s)g^{p'}(s)ds\right)^{1{/p'}}<\infty.
$$
Moreover
$$
\inf{\|f\|_{p,v} \|g\|_{w,v}} \le \|h\|_{\text{Ces}_p(w)}\le
2(p')^{1/{p'}}p^{1/p}\inf{\|f\|_{p,v}\|g\|_{w,v}},
$$
where the infimum is taken over all possible factorizations.  
\et

	Throughout this paper, we use standard notations and conventions. The letters $u$, $v$, $w$,..., are used for weight functions which are positive a.e. and locally integrable on $(0,\infty)$. The function $f$
	is real-valued and Lebesgue measurable on $(0,\infty)$. Also for a given weight $v$ we write $V(t)=\int_0^t v(s)ds$, $0\le t<\infty$.
	By $\chi_A$ we denote the
characteristic function of the measurable set $A$.

Observe that  the best known form of the Hardy inequality does not follow from the above result. 
The aim of this paper is to prove  factorization
results of the same type for the weighted Lebesgue, Ces\` aro and Copson spaces
of functions, which enhance in the same manner the weighted Hardy inequality. The weights satisfy the natural conditions which assure the boundedness of the Hardy, 
respectively the dual Hardy operators as well as some reversed conditions. 

 We denote by $P$  the Hardy operator and by $Q$ its
adjoint  \be\label{hardyop} Pf(t)=\frac{1}{t}\int_0^tf(x)dx; \quad
Qf(t)=\int_t^{\infty}\frac{f(x)}{x}dx, (t>0). \ee 
For $p\ge 1$, it is known that $P$ is bounded on the weighted
Lebesgue space $L^p(w)$ if and only if $w\in M_p$ (see \cite{M}), where $M_p$ is the
class of weights for which there exists  a constant $C>0$ such that, for all
$t>0$ it holds \be\label{mp} M_p:
\left(\int_t^{\infty}\frac{v(x)}{x^p}dx\right)^{1/p}\left(\int_0^{t}v^{1-p'}(x)dx\right)^{1/p'}\le
C. \ee

The least constant satisfying the condition $M_p$ will be denoted
 by $[v]_{M_p}$. Similarly, we denote by $m_p$ the class of weights
satisfying the reverse inequality and by $[v]_{m_p}$ the biggest
constant for which the reverse inequality holds.

 The Hardy operator $P$ is
	bounded on $L^1(v)$ if and only if there exists $C>0$, such that
	\be\label{ineq171} M_1: \int_t^{\infty}\frac{v(x)}{x}dx\le C v(t).
	\ee
	
	\noindent We denote by $[v]_{M_1}$ the least constant for which
	the above inequality is satisfied. Similarly, $[v]_{m_1}$ is the
	biggest constant for which the reverse inequality of
	(\ref{ineq171}) is satisfied.

The corresponding condition for the boundedness of the adjoint operator $Q$ on $L^p(v)$ (see \cite{M}) is given by
\begin{equation}\label{mpstar}
M^*_p:
\left(\int_0^{t}v(x)dx\right)^{1/p}\left(\int_t^{\infty}\frac{v^{1-p'}(x)}{x^{p'}}dx\right)^{1/p'}\le
C.
\end{equation}
The least constant satisfying   the $M^*_p$ condition will be denoted
$[w]_{M^*_p}$. Similarly, we denote by $m_p^*$ the class of weights
satisfying the reverse inequality and by $[w]_{m_p^*}$ the biggest
constant for which the reverse inequality holds.

The dual Hardy operator $Q$, (defined by (\ref{hardyop})) is
bounded on $L^1(v)$ if and only if there exists $C>0$, such that
\be\label{ineq17} M_1^*: \frac{1}{t}\int_0^{t}v(x)dx\le C v(t).
\ee
We denote by $[v]_{M_1^*}$ the least constant for which
the above inequality is satisfied. Similarly, $[v]_{m_1^*}$ is the
biggest constant for which the reverse inequality of
(\ref{ineq17}) is satisfied. 

In Section 2 we prove a factorization result for the weighted Lebesgue spaces $L^p(v)$. This result is a natural extension of Theorem 3.8 from \cite{B}.

In Section 3 we present some factorization theorems for the weighted
Ces\` aro spaces in terms of weighted Lebesgue spaces and the spaces $G_p(v)$,
for $p>1$. We treat  separately  the case $p=1$ which appears to be new.  Moreover, our study
is motivated by similar factorization results established by G.
Bennett \cite{B}, in the unweighted case, for
spaces of sequences. As a consequence we recover the best known form of the Hardy inequality for weighted Lebesgue spaces.

We also present the optimal result for the power weights. Section 4  is
devoted to the same problems but for Copson spaces. 

\section{The spaces $D_p(v)$ and $G_p(v)$}

For $p>0$, the function spaces  $G_p(v)$ and $D_p(v)$ are defined by
\be\label{gp}
 G_p(v)=\{f :\sup_{t>0}\left(\frac{1}{V(t)}\int_0^t|f(x)|^pv(x)dx\right)^{1/p}<\infty\}
\ee
 and
\be\label{dp} D_p(v)=\{f : \left(\int_0^{\infty}\text{esssup}_{t\ge
x}|f(t)|^pv(x)dx\right)^{1/p} <\infty \}. \ee

These spaces are Banach spaces, for $p\ge 1$,
endowed with the norms \be
\|f\|_{G_p(v)}=\sup_{t>0}\left(\frac{1}{V(t)}\int_0^t|f(x)|^pv(x)dx\right)^{1/p},\ee
respectively \be \|f\|_{D_p(v)}=\left(\int_0^{\infty}\text{esssup}_{t\ge
x}|f(t)|^pv(x)dx\right)^{1/p}.\nonumber \ee

We denote by $\widehat{f}(x)={\text{essup}}_{t\ge x}|f(t)|$ the least
decreasing majorant of the absolute value of the function $f$. Obviously, the
function $f\in D_p(v)$ if and only if $\widehat{f}\in L_p(v)$ and
that $\|f\|_{D_p(v)}=\|\widehat{f}\|_{L_p(v)}$. 

In what follows we need the following two lemmas.
\bl[Hardy's lemma]\label{Hardy} Let $f,g$ be two nonnegative
real-valued functions and $h$ be a nonnegative decreasing
function. If
$$
\int_0^tf(x)dx\le \int_0^tg(x)dx, \quad \text{for any } t>0
$$
then
$$
\int_0^\infty f(x)h(x)dx\le \int_0^\infty g(x)h(x)dx.
$$
 \el
\begin{proof}
See \cite[Proposition 3.6]{BS}.
\end{proof}

\bl\label{level} Let $h$ be a nonnegative measurable function on
$(0,\infty),$ such that
$$
\lim_{x\rightarrow \infty}\frac{\int_0^xh(t)v(t)dt}{\int_0^x
v(t)dt}=0.
$$
Then there exists a nonnegative decreasing function $h^{\circ}$  on $(0,\infty)$, called the
level function of $h$  with respect to the measure $v(x)dx$ satisfying the following conditions:
 \begin{enumerate}
    \item $\int_0^x h(t)v(t)dt\le \int_0^x h^{\circ}(t)v(t)dt;$
    \item up to a set of measure zero, the set $\{x:h(x)\ne
    h^{\circ}(x)\}=\cup_{k=1}^{\infty}I_k$, where $I_k$ are bounded disjoint intervals such that
    $$
\int_{I_k}h^{\circ}(t)v(t)dt=\int_{I_k}h(t)v(t)dt
    $$
    and $h^{\circ}$ is constant on $I_k$, i.e. $h^{\circ}(t)=\frac{\int_{I_k}hv}{\int_{I_k}v}.$
\end{enumerate}
\el
\begin{proof}
For a proof see e.g. \cite{BKS} or \cite{S}. 
\end{proof}
We are now ready to prove the main theorem of this section which contains one of the possible factorizations of weighted Lebesgue spaces.  

\bt If $0<p\le\infty$, then
 a function $h\in L^p(v)$ if and only if $f$ admits a factorization $h=f\cdot g$ such that $f\in D_p(v) $ and $g\in G_p(v)$.
Moreover,
$$
\|h\|_{L^p(v)}=\inf\{ \|f\|_{D_p(v)}\|g\|_{G_p(v)}\},
$$ where the infimum is taken over all possible factorizations $h=f\cdot g$ with $f\in D_p(v) $ and $g\in G_p(v)$. \et
\begin{proof}
The case $p=\infty$ is trivial. By homogeneity, it is sufficient
to prove the theorem for $p=1$. 

We first prove that $D_1(v)\cdot
G_1(v) \subseteq L^1{(v)}$. Suppose that $h$ admits a factorization
$h=f\cdot g$ with $f\in D_1(v)$, $g\in G_1(v)$. Then $$
\|h\|_{L^1(v)}=\int_0^\infty |f(x)g(x)|v(x)dx\le
\int_0^{\infty}\widehat{f}(x)|g(x)|v(x)dx. $$
From definition
(\ref{gp}) we have the inequality
$$
\int_0^t |g(x)|v(x)dx\le \|g\|_{G_1(v)}\int_0^t v(x)dx, \quad \text{for any } t>0
$$
which together with Lemma \ref{Hardy} give
$$
\int_0^{\infty}\widehat{f}(x)|g(x)|v(x)dx
\le\|g\|_{G_1(v)}\int_0^{\infty}\widehat{f}(x)v(x)dx=\|g\|_{G_1(v)}\|f\|_{D_1(v)}.
$$
Thus we have that  $D_1(v)\cdot G_p(v)\subseteq L^1(v)$ and
that
$$\|h\|_{L^1(v)}\le\inf\{\|g\|_{G_1(v)}\|f\|_{D_1(v)}\},$$ where
the infimum is taken over all possible factorizations $h=f\cdot g$.

Conversely   let $h$ be a nonnegative function such that $h\in L^1(v)$.
We set $f(x)=h^{\circ}(x)$, $x>0$, where $h^{\circ}(x)$ is the level
function of $h$ with respect to the measure $v(x)dx$, as in Lemma
\ref{level}. Since $h^{\circ}(x)$ is a decreasing function by the
definition of the space $D_1(v)$ and by Lemma \ref{level} we have
that
$$
\|f\|_{D_1(v)}=\|h^{\circ}\|_{D_1(v)}=\|h^{\circ}\|_{L_1(v)}=\|h\|_{L_1(v)}.
$$
 We define $g(x)=\frac{h(x)}{h^{\circ}(x)}$ for any $x>0$. If $t\in I_n$, for some $n$, we have
\begin{align*}
\frac{1}{V(t)}\int_0^t g(x)v(x)dx&=\frac{1}{V(t)}
\int_0^t\frac{h(x)}{h^{\circ}(x)}v(x)dx\\
&=\frac{1}{V(t)}\left(\int_E\frac{h(x)}{h^{\circ}(x)}v(x)dx+\int_{\cup_{k=1}^{n-1}I_k}\frac{h(x)}{h^{\circ}(x)}v(x)dx\right.\\
&\left.+ \int_{a_n}^t \frac{h(x)}{h^{\circ}(x)}v(x)dx\right),
\end{align*}
where $E=\{x\in (0,t): h(x)=h^{\circ}(x)\}$ and $I_k =(a_k,b_k)$  
 are the disjoint intervals from Lemma \ref{level}. Hence, by
Lemma \ref{level} we get that
$$
\int_E\frac{h(x)}{h^{\circ}(x)}v(x)dx=\int_E v(x)dx,
$$
$$\int_{I_k} \frac{h(x)}{h^{\circ}(x)}v(x)dx= \int_{I_k} v(x)dx$$
and 
$$
\int_{a_n}^t \frac{h(x)}{h^{\circ}(x)}v(x)dx\le \int_{I_n} v(x)dx.
$$
Hence
$$
\|g\|_{G_1(v)}\le 1.
$$
Since $h=f\cdot g$, with $f\in D_1(v)$ and $g\in G_1(v)$ we have
that $L^1(v)\subseteq D_1(v)\cdot G_1(v)$ and
$$
\|h\|_{L^1(v)}=\|f\|_{D_1(v)}\ge
\|f\|_{D_1(v)}\cdot \|g\|_{G_1(v)}\ge\inf\{
\|f\|_{D_1(v)}\cdot \|g\|_{G_1(v)}\},
$$
where the infimum is taken over all possible factorizations
$h=f\cdot g$. It is easy to see from this proof that the infimum is actually attained and this concludes the proof of the theorem.
\end{proof}

\section{Factorization of the weighted Ces\` aro spaces}

In this section we present a factorization of the weighted Ces\` aro
spaces $\text{Ces}_p(v)$.  We treat separately the cases $p>1$ and $p=1$. The weighted Ces\` aro spaces of 
functions, ${\text{Ces}}_p(v)$  is defined to be the space of all Lebesgue measurable real functions on $[0,\infty)$ such that 
$$
\|f\|_{\text{Ces}_p(v)}=
\left(\int_0^{\infty}\left(\frac{1}{x}\int_0^x|f(t)|dt\right)^pv(x)dx\right)^{1/p}<\infty.
$$
These spaces are obviously Banach spaces, for $p\ge
1$ and if the weight $v$ satisfies (\ref{mp}) we have that $L^p(v)\subseteq \text{Ces}_p(v)$. We denote by
\be\label{fact1}!h!_{p,v}= \inf\{\|f\|_{L^p(v)}\|g\|_{G_{p'}(v^{1-p'})} \}\ee
 where
the infimum is taken over all possible decompositions of $h=f\cdot
g$, with $f\in L^p(v)$ and $g\in G_{p'}(v^{1-p'})$.

 \bt \label{mains2} Let $p>1$ and $v$ belongs to the class $M_p$ and $m_p$.
The function $h$ belongs to $Ces_p(v)$ if and only if it admits a
factorization $h=f \cdot g$, with $f\in L^p(v)$ and $g\in
G_{p'}(v^{1-p'})$. Moreover
$$
[v]_{m_p}!h!_{p,v} \le \|h\|_{\text{Ces}_p(v)}\le
{(p')}^{1/{p'}}p^{1/p}[v]_{M_p}!h!_{p,v}.
$$
\et
\begin{proof}
Let $f\in L^p(v)$ and $g\in G_{p'}(v^{1-p'})$.  First we prove that
 the function $h=f\cdot g\in {\text{Ces}}_p(v)$ and the right-hand side inequality. Let $u$
be an arbitrary decreasing function. By H\"{o}lder's
inequality we get
 \be\label{eq11}
 \int_0^t |h(x)|dx=\int_0^t |f(x)g(x)|dx\le
 \ee
 \be\nonumber \left(\int_0^t
|f(x)|^pv(x)u^{-p}(x)dx\right)^{1/p}\left(\int_0^t
|g(x)|^{p'}v^{1-p'}(x)u^{p'}(x)dx\right)^{1/p'}. \ee

\noindent On the other hand, by Lemma \ref{Hardy} we obtain
 \begin{align}\label{eq12}
\left(\int_0^t |g(x)|^{p'}v^{1-p'}(x)u^{p'}(x)dx\right)^{1/p'}&\le
\|g\|_{G_{p'}(v^{1-p'})}\\
&\cdot \left(\int_0^tv^{1-p'}(x)u^{p'}(x)dx\right)^{1/p'}{\notag}.
\end{align}
Hence, by (\ref{eq11}) and (\ref{eq12}), integrating from 0 to
$\infty$ and by applying Fubbini's theorem  we have
 \begin{align*}
\int_0^{\infty}\left(\frac{1}{t}\int_0^{t}|h(x)|dx\right)^p & v(t)dt
\le\|g\|^p_{G_{p'}(v^{1-p'})}\int_0^{\infty}\left(\int_0^{t}|f(x)|^p
v(x) u^{-p}(x) dx\right)\\
&\cdot \left(\int_0^{t}v^{1-p'}(x)u^{p'}(x)dx \right)^{p-1}t^{-p}v(t)dt\\ 
&=\|g\|^p_{G_{p'}(v^{1-p'})}\int_0^{\infty}|f(x)|^pv(x)u^{-p}(x)\\
&\cdot \left(\int_x^{\infty}t^{-p}v(t)\left(\int_0^{t}v^{1-p'}(x)u^{p'}(x)dx\right)^{p-1}dt\right)
  dx. 
 \end{align*}
 
\noindent Taking $u(t)=\left(\int_0^t
v^{1-p'}(s)ds\right)^{-1/{(pp')}}$, since
$v\in M_p$ we get
$$
\|h\|_{Ces_p(v)}\le
p^{1/p}{(p')}^{1/{p'}}[v]_{M_p}\|g\|_{G_{p'}(v^{1-p'})}\|f\|_{L^{p}(v)}.
$$
 Hence $h\in {\text{Ces}}_p(v)$ and
$$\|h\|_{{\text{Ces}}_p(v)}\le
p^{1/p}{(p')}^{1/{p'}}[v]_{M_p}\inf\{\|f\|_{L^{p}(v)}\cdot \|g\|_{G_{p'}(v^{1-p'})}\},$$
where  the infimum is taken over all possible factorizations of
$h$.  This completes the first part of the proof of the theorem.

For the reversed embedding, i.e. ${\text{Ces}}_p(v)\subseteq
L^p(v)\cdot G_{p'}(v^{1-p'})$, let $h\in {\text{Ces}}_p(v)$ and set 
$w(t):=\frac{1}{v(t)}\int_t^{\infty}\frac{v(x)}{x}\left(\frac{1}{x}\int_0^x
|h(s)|ds\right)^{p-1}dx$, $t>0$. We may assume, without loss of generality
that $v(t)> 0$, for any $t>0$. We define
$f(t)=|h(t)|^{1/p}w^{1/p}(t)\sign h(t)$ and
$g(t)=|h(t)|^{1/p'}w^{-1/p}(t)$. It is easy to see that \be
 \|f\|_{L^p(v)}=\|h\|_{{\text{Ces}}_p(v)}<\infty. \ee
By H\"{o}lder's inequality we have
 \begin{align}\label{eq13}
 \left(\int_0^t g^{p'}(x)v^{1-p'}(x)dx\right)^p&\le
\left(\int_0^t |h(x)|dx\right)^{p-1}\\
&\cdot\left(\int_0^t
|h(x)|w^{-p'}(x)v^{-p'}(x)dx\right) \notag.
\end{align}
 Multiplying the inequality
(\ref{eq13}) by $\int_t^{\infty}x^{-p}v(x)dx$   and using that
$w(t)v(t)$ is a decreasing function we get

\begin{align*}
\left( \int_t^{\infty}\frac{v(x)}{x^p}dx\right) &\left(\int_0^t
g^{p'}(x)v^{1-p'}(x)  dx\right)^p
 \le\left(\int_t^{\infty}\frac{v(x)}{x^p}\left(\int_0^x
|h(s)|ds\right)^{p-1}dx\right)\\
&\cdot \left(\int_0^t
|h(x)|w^{-p'}(x)v^{-p'}(x)dx\right)\\
&=w(t)v(t)\left(\int_0^t
|h(x)|w^{-p'}(x)v^{-p'}(x)dx\right)\\
&\le \int_0^t
|h(x)|w^{1-p'}(x)v^{1-p'}(x)dx.
\end{align*}

 Since
$g^{p'}(x)=|h(x)|w^{1-p'}(x)$ we obtain

\begin{align*}
\left(\frac{1}{\int_0^t v^{1-p'}(x)dx }\int_0^t
g^{p'}(x)v^{1-p'}(x)dx\right)^{1/p'}&\le 
\left(\int_0^t
v^{1-p'}(x)dx)\right)^{-1/p'}\\
&\cdot\left(\int_t^{\infty}\frac{v(x)}{x^p}dx\right)^{-1/p}.
\end{align*}
Hence \be\nonumber
 \sup_{t>0}\left(\frac{1}{\int_0^t
v^{1-p'}(x)dx }\int_0^t g^{p'}(x)v^{1-p'}(x)dx\right)^{1/p'}
\le\frac{1}{[v]_{m_p}} \ee
which shows that $g$ belongs to
$G_{p'}(v^{1-p'})$  and
$$\|h\|_{\text{Ces}_p(v)}=\|f\|_{L_p(v)}\ge
[v]_{m_p}\|f\|_{L_p(v)}\|g\|_{G_{p'}(v^{1-p'})}.$$
 In this way, we
get the left-hand side inequality.

\end{proof}

If we take $g(x)=1$, $x>0$ the right-hand side inequality implies the best form of the weighted Hardy inequality
 for $1<p<\infty$ see e.g. \cite{KMP}.
 
 Observe also that the infimum is attained.

 In particular,  we
denote by $L_{\alpha}^p$ the weighted Lebesgue space with the power
weight $v(t)=t^{\alpha}$ and in a similar way the spaces
$G_{p,\alpha}$ and $\text{Ces}_{p,\alpha}$. In analogy with the
general case we also denote  by
$$!h!_{p,\alpha}=\inf\|f\|_{L^p_{\alpha}}\|g\|_{G_{p',{\alpha(1-p')}}},$$
where the infimum is taken over all possible decompositions of
$h=f\cdot g$, with $f\in {L^p_{\alpha}}$ and $g\in
{G_{p',{\alpha(1-p')}}}$.
\begin{corollary} Let $p>1$ and $-1<\alpha<p-1$.
The function $h$ belongs to $\text{Ces}_{p,\alpha}$ if and only if
it admits a factorization $h=f \cdot g$, with $f\in
{L^p_{\alpha}}$ and $g\in {G_{p',{\alpha(1-p')}}}$. Moreover
$$
\left(\frac{1}{p}\right)^{1/p}\left(\frac{1}{p'}\right)^{1/{p'}}
\frac{p}{p-\alpha-1}!h!_{p,\alpha} \le \|h\|_{Ces_{p,\alpha}}\le
\frac{p}{p-\alpha-1}!h!_{p,\alpha}.
$$
\end{corollary}
\begin{proof}
Take $v(t)=t^{\alpha}$ in Theorem \ref{mains2}. The constant in
the right hand-side inequality is optimal since it is the best
constant in Hardy's inequality with a power weight (see e.g \cite{KMP} p. 23).

\end{proof}

For the sake of completeness, as well as for the  independent interest we
present separately the case $p=1$, although the proof of the main
result in this case follows the same ideas as for $p>1$.

By $L^{\infty}$ we denote, as usual,
the space of all measurable functions which satisfy the condition
$$
\|g\|_{\infty}:=\text{esssup}_{x>0}|g(x)|<\infty.
$$
As before, $$!h!_{1,v}= \inf\|f\|_{L^1(v)}\|g\|_{\infty}$$
where the infimum is taken over all possible factorizations of $h=f\cdot
g$, with $f\in L^1(v)$ and $g\in L^{\infty}$.

\bt\label{teoremweak} Let $v$ belong to $M_1$ and $m_1$. The function $h$ belongs to
${\text{Ces}}_1(v)$ if and only if it admits a factorization $h=f
\cdot g$, with $f\in L_1(v)$ and $g\in G_{\infty}$. Moreover
$$
[v]_{m_1}!h!_{1,v}\le \|h\|_{{\text{Ces}_1}{(v)}}\le
[v]_{M_1}!h!_{1,v}.
$$
\et
\begin{proof}
Let $f\in L_1(v)$ and $g\in L^{\infty}$. We prove that the
function $h=fg$ belongs to ${\text{Ces}}_1(v)$. By H\"{o}lder's
inequality and since $g\in L^{\infty}$ we get
 \be\label{eq1011}
\int_0^{\infty}\left(\frac{1}{t}\int_0^{t}h(s)ds\right)v(t)dt
\le\|g\|_{{\infty}}\int_0^{\infty}\left(\frac{1}{t}\int_0^{t}f(x)dx
\right)v(t)dt. \ee

\noindent By Fubini's theorem and taking into account that $v\in
M_1$ we have that
$$
\|h\|_{{\text{Ces}}_1(v)}\le
[v]_{M_1}\|g\|_{{\infty}}\|f\|_{L_{1}(v)},
$$
for any $f$, $g$ as above. Hence $h\in \text{Ces}_1(v)$ and
$$\|h\|_{Ces_1(v)}\le
[v]_{M_1}\inf\|g\|_{{\infty}}\|f\|_{L_{1}(v)},$$ where infimum
is taken over all possible factorizations of $h$.  This completes
the first part of the proof.

Conversely, let $h\in {\text{Ces}}_1(v)$ and
$w(t)=\frac{1}{v(t)}\int_t^{\infty}\frac{v(x)}{x}dx$. We may assume, without loss of generality that $v(t)>0$, for all $t>0.$\\
Let $f(t)=|h(t)|w(t)\sign h(t)$ and $g(x)=\frac{1}{w(x)}$. It is
easy to see that \be \nonumber
 \|f\|_{L_1(v)}=\|h\|_{{\text{Ces}}_1(v)}<\infty.
  \ee

\noindent Since $v\in m_1$,  $g$ belongs to $L_{\infty}$ and
\be\nonumber\|g\|_{{\infty}}\le \frac{1}{[v]_{m_1}}. \ee
Moreover, $\|h\|_{{\text{Ces}}_1(v)}=\|f\|_{L_1(v)}\ge
[v]_{m_1}\|f\|_{L_1(v)}\|g\|_{{\infty}}$ and we get the
left-hand side inequality of the theorem. The proof is complete.
\end{proof}

\section{Factorization of the weighted Copson spaces}

In the same manner, in this section we present the factorizations of the weighted
Copson space, namely the space
$$
{\text{Cop}}_p(v)=\{f:\int_0^{\infty}\left(\int_t^{\infty}\frac{|f(x)|}{x}dx\right)^pv(t)dt<\infty\}.
$$
Let \be\label{gstarp}
 G^{*}_p(v)=\{f :\sup_{t>0}\left(\frac{1}{\int_t^{\infty}v(x)x^{-p}dx}\int_t^{\infty}f(x)^pv(x)x^{-p}dx\right)^{1/p}<\infty\}
\ee
To prove the main result we need the following Lemma.

\bl\label{Hardy1}  Let $f,g$ be two non-negative real-valued
functions and $h$ be a non-negative increasing function. If
$$
\int_t^{\infty}f(x)dx\le \int_t^{\infty}g(x)dx, \quad t>0
$$
then
$$
\int_t^{\infty}f(x)h(x)dx\le \int_t^{\infty}g(x)h(x)dx.
$$
 \el
\begin{proof}
The proof follows by a change of variable and Lemma \ref{Hardy}.
\end{proof}
We denote by
\be\label{exc}
!!h!!_{p,v}=\inf \|f\|_{\text{L}_p(v)} \|g\|_{\text{G}^{*}_{p'}(v^{1-p'})}
\ee
where the infimum is taken over all possible factorizations of $h=f\cdot g.$
\bt \label{mains3} Let $p>1$ and $v$ belong to the class $M_p^*$
and $m_p^*$.

The function $h$ belongs to $\text{Cop}_p(v)$ if and only if it
admits a factorization $h=f \cdot g$, with $f\in L^p(v)$ and $g\in
\text{G}^{*}_{p'}(v^{1-p'})$. Moreover
$$
[v]_{m_p^*}!!h!!_{p,v} \le \|h\|_{\text{Cop}_p(v)}\le
{p'}^{1/{p'}}p^{1/p}[v]_{M_p^*}!!h!!_{p,v}.
$$
\et

\begin{proof}
Let $f\in L^p(v)$ and $g\in \text{G}^{*}_{p'}(v^{1-p'}) $.

We show
first that the function $h=fg\in {\text{Cop}}_p(v)$. Let $u$ be an
arbitrary positive increasing function. H\"{o}lder's inequality
gives
 \be\label{eq101}
 \int_t^{\infty} \frac{f(x)g(x)}{x}dx\le
 \ee
 \be\nonumber \left(\int_t^{\infty}
f^p(x)v(x)u^{-p}(x)dx\right)^{1/p}\left(\int_t^{\infty}
g^{p'}(x)\frac{v^{1-p'}(x)}{x^{p'}}u^{p'}(x)dx\right)^{1/p'}. \ee

\noindent By  Hardy's Lemma \ref{Hardy1} we obtain
 \begin{align*}
\left(\int_t^{\infty}
g^{p'}(x)\frac{v^{1-p'}(x)}{x^{p'}}u^{p'}(x)dx\right)^{1/p'}&\le
 \|g\|_{\text{G}^{*}_{p'}(v^{1-p'})}^p\\
 & \cdot\left(\int_t^{\infty}\frac{v^{1-p'}(x)}{x^{p'}}u^{p'}(x)dx\right)^{1/p'}.
\end{align*}

\noindent Hence, multiplying  (\ref{eq101}) by $v(t)$, raising to $p$
and integrating from 0 to $\infty$,  we get
 \begin{align*}
\int_0^{\infty}&\left(
\int_t^{\infty}\frac{h(x)}{x}dx\right)^pv(t)dt \le\|g\|^p_{G_{p'}({v^{1-p'}})}\\
&\cdot \int_0^{\infty}v(t)\left(\int_t^{\infty}f^p(x)
v(x)u^{-p}(x)dx\right)
\left(\int_t^{\infty}\frac{v^{1-p'}(x)}{x^{p'}}u^{p'}(x)dx\right)^{p-1}dt.
\end{align*}

\noindent By Fubini's theorem we have
\begin{align*}
\int_0^{\infty}&\left(\int_t^{\infty}\frac{h(s)}{s}ds\right)^pv(t)dt
\le\|g\|^p_{G_{p'}({v^{1-p'}})}\\
&\cdot\int_0^{\infty}f^p(x)v(x)
\left(\int_0^xv(t)
\left(\int_t^{\infty}\frac{v^{1-p'}(s)}{s^{p'}}u^{p'}(s)ds\right)^{p-1}dt\right)u^{-p}(x)dx.
\end{align*}

\noindent Taking
$u(t)=\left(\int_t^{\infty}\frac{v^{1-p'}(x)}{x^{p'}}\right)^{-1/{pp'}}$,
in the above inequality and
 since
$$
\int_t^{\infty}\frac{v^{1-p'}(s)}{s^{p'}}u^{p'}(s)ds=p'\left(\int_t^{\infty}\frac{v^{1-p'}(s)}{s^{p'}}(s)ds\right)^{1/p'}dx
$$
we have that

\begin{align*}
\int_0^{\infty}f^p(x)&v(x) \left(\int_0^xv(t)
\left(\int_t^{\infty}\frac{v^{1-p'}(s)}{s^{p'}}u^{p'}(s)ds\right)^{p-1}dt\right)u^{-p}(x)dx\\
&=(p')^{p-1}\int_0^{\infty}f^p(x)v(x)\left(\int_0^x v(t)
\left(\int_t^{\infty}\frac{v^{1-p'}(s)}{s^{p'}}ds\right)^{\frac{p-1}{p'}}dt\right)\\
&\cdot\left(\int_x^{\infty}\frac{v^{1-p'}(s)}{s^{p'}}ds\right)^{1/p'}.
\end{align*}

By the definition of $M_p^*$ we get
\begin{align*}
&\int_0^{\infty}\left(\int_t^{\infty}\frac{h(s)}{s}ds\right)^pv(t)dt
\le (p')^{p-1}\|v\|^{p-1}_{M_p^*}\\
& \cdot\int_0^{\infty}f^p(x)v(x)\left(\int_0^x
v(t)\left(\int_0^tv(s)ds\right)^{-1/{p'}}dt\right)\left(\int_x^{\infty}\frac{v^{1-p'}(s)}{s^{p'}}ds\right)^{1/p'}dx\\
&=(p')^{p-1}p\|v\|^{p-1}_{M_p^*}\int_0^{\infty}f^p(x)v(x)\left(\int_0^x
v(t)\right)^{1/p}\left(\int_x^{\infty}\frac{v^{1-p'}(s)}{s^{p'}}ds\right)^{1/p'}dx\\
&\le(p')^{p-1}p\|v\|^{p}_{M_p^*}\int_0^{\infty}f^p(x)v(x)dx,
\end{align*}
since
$$
\frac{d}{dx}\left(\int_0^xv(t)dt\right)^{1/p}=\frac{1}{p}\left(\int_0^xv(t)dt\right)^{1/p-1}v(x).
$$

\noindent Hence
$$
\|h\|_{\text{Cop}_p(v)}\le
p^{1/p}{p'}^{1/{p'}}[v]_{M_p^*} \|g\|_{\text{G}^{*}_{p'}(v^{1-p'})}\|f\|_{L_{p}(v)}.
$$
for any $f$, $g$ as above.
Hence $h\in {\text{Cop}}_p(v)$ and
$$\|h\|_{{\text{Cop}}_p(v)}\le p^{1/p}{p'}^{1/{p'}}[v]_{M_p^*}\inf \|g\|_{\text{G}^{*}_{p'}(v^{1-p'})}\|f\|_{L_{p}(v)}$$
where  the infimum is taken over all possible factorizations of
$h$ which gives the left-hand side inequality of the theorem.

For the reverse embedding, i.e. ${\text{Cop}}_p(v)\subset
L_p(v)G^*_{p'}(v^{1-p'})$, let $h\in {\text{Cop}}_p(v)$ and
$w(t):=\frac{1}{tv(t)}\int_0^{t}v(x)\left(\int_{x}^{\infty}\frac{h(s)}{s}ds\right)^{p-1}dx
$, if $v\ne 0$. Define $f(t)=|h|^{1/p}(t)w^{1/p}(t)\sign h(t)$ and
$g(t)=|h|^{1/p'}(t)w^{-1/p}(t)$. An easy application of Fubini
theorem gives \be\nonumber
 \|f\|_{L^p(v)}=\|h\|_{{\text{Cop}}_p(v)}<\infty. \ee
By H\"{o}lder's inequality and the definition of $g$ we have
 \be\label{eq103}
 \left(\int_x^{\infty} g^{p'}(t)\frac{v^{1-p'}(t)}{t^{p'}} dt\right)^p\le
\left(\int_x^{\infty} \frac{|h(t)|}{t}dt\right)^{p-1}
\int_x^{\infty}\frac{|h(t)|}{t}w^{-p'}(t)\frac{v^{-p'}(t)}{t^{p'}}dt.\ee

\noindent We estimate first the right-hand side term of the inequality
(\ref{eq103}) multiplied by $\int_0^{x}v(t)dt$.

\be\nonumber \left(\int_0^x v(t)dt\right)\left(\int_x^{\infty}
\frac{|h(t)|}{t}dt\right)^{p-1}
\int_x^{\infty}\frac{|h(t)|}{t}w^{-p'}(t)\frac{v^{-p'}(t)}{t^{p'}}dt
\ee

\be\nonumber \le \int_0^xv(t)\left(\int_t^{\infty}
\frac{|h(s)|}{s}ds\right)^{p-1}dt
\int_x^{\infty}\frac{|h(t)|}{t}w^{-p'}(t)\frac{v^{-p'}(t)}{t^{p'}}dt
\ee

\be\nonumber =xw(x)v(x)\left(\int_x^{\infty}
h(t)w^{-p'}(t)v^{-p'}(t)dt\right)\le \int_x^\infty
h(t)w^{1-p'}(t)v^{1-p'}(t)t^{-p'}dt, \ee since, by definition,
$xw(x)v(x)$ is an increasing function. By using that
$g^{p'}(x)=h(x)w^{1-p'}(x)$ we get
$$
\left(\frac{1}{\int_x^\infty \frac{v^{1-p'}(t)}{t^{p'}}dx
}\int_x^\infty g^{p'}(t)\frac{v^{1-p'}(t)}{t^{p'}}dt\right)^{1/p'}
$$
$$\le\frac{1}{\left(\int_0^x v(t)dt
\right)^{1/p}\left(\int_x^{\infty}t^{-p'}v^{1-p'}(t)dt\right)^{1/p'}}.
$$
Hence \be\label{last}
 \sup_{t>0}\left(\frac{1}{\int_t^{\infty}
v^{1-p'}(x)x^{-p'}dx }\int_t^{\infty} g^{p'}(x)v^{1-p'}(x)x^{-p'}dx\right)^{1/p'}
\le\frac{1}{[v]_{m_p^*}} \ee
which means that $g$ belongs to $ {\text{G}^{*}_{p'}(v^{1-p'})}$. Moreover,\\
$\|h\|_{\text{Cop}_p(v)}=\|f\|_{L^p(v)}\ge
[v]_{m_p^*}\|f\|_{L^p(v)} \|g\|_{\text{G}^{*}_{p'}(v^{1-p'})}$. In this way the left-hand side inequality is proved.  

\end{proof}

\noindent The space $\text{Cop}_{p,\alpha}$ is the Copson weighted space with the weight $t^{\alpha}$. We have 
 the following result for the case of a power weight. 
\begin{corollary} Let $p>1$ and $\alpha>-1$.
The function $h$ belongs to $\text{Cop}_{p,\alpha}$ if and only if
it admits a factorization $h=f \cdot g$, with $f\in
{L^p_{\alpha}}$ and $g\in {G_{p',{\alpha(1-p')}}}$. Moreover
$$
\frac{(p-1)^{1/p'}}{\alpha+1}!h!_{p,\alpha} \le \|h\|_{{p,\alpha}}\le
\frac{p}{\alpha+1}!h!_{p,\alpha}.
$$  The constants in both inequalities are optimal.
\end{corollary}
\begin{proof}
Take $v(t)=t^{\alpha}$ in Theorem \ref{mains3}. The constant in
the right hand-side inequality is optimal since it is the best
constant in Hardy inequality (see e.g \cite{KMP}) and the optimality of the constant in the left-hand side follows if we take $h(x)=\chi_{(a-\varepsilon,a+\varepsilon)}$ and let then $\varepsilon\rightarrow 0$ 
and $a \rightarrow\infty$.
\end{proof}

\noindent We present now the case $p=1$.

\noindent The dual Hardy operator $Q$, (defined by (\ref{hardyop})) is
bounded on $L^1(v)$ if and only if there exists $C>0$, such that
\be\label{ineq_ultima} M_1^*: \frac{1}{t}\int_0^{t}v(x)dx\le C v(t).
\ee

\noindent We denote by $[v]_{M_1^*}$ the least constant for which
the above inequality is satisfied. Similarly, $[v]_{m_1^*}$ is the
biggest constant for which the reverse inequality of
(\ref{ineq_ultima}) is satisfied. 

\bt Let $v$ belong to $M_1^*$ and $m_1^*$. The function $h$ belongs to
${\text{Cop}}_1(v)$ if and only if it admits a factorization $h=f
\cdot g$, with $f\in L_1(v)$ and $g\in G_{\infty}$. Moreover
$$
[v]_{m_1^*}!h!_{1,v}\le \|h\|_{{\text{Cop}_1}{(v)}}\le
[v]_{M_1^*}!h!_{1,v}.
$$
and the constants are optimal. 
\et
\begin{proof}
 The proof is similar with that of Theorem ~\ref{teoremweak}.
\end{proof}

\end{document}